\documentclass[12pt]{article}

\usepackage{tabularx}
\usepackage{tikz}
\usepackage{relsize}
mm \textheight=8.9in

\usepackage{amssymb}
\usepackage{amsmath}
\usepackage{amsthm}
\usepackage{color}

\newcommand{\hgt}{\mathrm{ht}}

\newcommand{\Qbb}{\mathbb{Q}}

\newcommand{\Sc}{\mathcal{S}}

\newcommand{\br}[3]{{$#1$}$\lower4pt\hbox{$\tp\atop\raise4pt \hbox{$\scriptscriptstyle{#2}$}$} ${$#3$}}
\newcommand{\tw}[3]{{$#1$}${\,\scriptscriptstyle {#2}}\atop\raise9pt\hbox{$\scriptstyle\tp$} ${$#3$}}
\newcommand{\ttps}[2]{{#1}\raise5pt\hbox{$\lower12pt\hbox{$\scriptstyle\tp$}\atop \lower0pt\hbox{$\tilde\;$}$}\raise4.5pt\hbox{${\scriptstyle{#2}}$}}
\newcommand{\st}[1]{\mbox{${\,\scriptscriptstyle {#1}}\atop\raise5.5pt\hbox{$*$}$}}

\newcommand{\rd}[1]{\mbox{${\,\scriptscriptstyle {#1}}\atop\raise5.5pt\hbox{$\bullet$}$}}
\newcommand{\rt}[1]{\otimes_\chi}
\newcommand{\lt}[1]{\mbox{${\,\scriptscriptstyle {#1}}\atop\raise5.5pt\hbox{$\ltimes$}$}}
\newcommand{\btr}{\raise1.2pt\hbox{$\scriptstyle\blacktriangleright$}\hspace{2pt}}
\newcommand{\btl}{\raise1.2pt\hbox{$\scriptstyle\blacktriangleleft$}\hspace{2pt}}

\newcommand{\lcr}{\raise1.0pt \hbox{${\scriptstyle\rightharpoonup}$}}
\newcommand{\rcr}{\raise1.0pt \hbox{${\scriptstyle\leftharpoonup}$}}

\newcommand{\ttp}{{\lower12pt\hbox{$\tp$}\atop \hbox{$\tilde\;$}}}

\newcommand{\id}{\mathrm{id}}

\newcommand{\wt}{\mathrm{wt}}

\newcommand{\Ac}{\mathcal{A}}

\newcommand{\Jc}{\mathcal{J}}

\newcommand{\B}{{B}}

\newcommand{\Hg}{\mathfrak{H}}

\newcommand{\Ru}{\mathcal{R}}
\newcommand{\Zc}{\mathcal{Z}}

\newcommand{\Vc}{\mathcal{V}}

\newcommand{\Zg}{\mathfrak{Z}}

\renewcommand{\O}{\mathcal{O}}

\newcommand{\C}{\mathbb{C}}
\newcommand{\Z}{\mathbb{Z}}

\newcommand{\N}{\mathbb{N}}

\newcommand{\tp}{\otimes}

\newcommand{\U}{U}

\newcommand{\Fc}{\mathcal{F}}

\newcommand{\gm}{\gamma}
\newcommand{\dt}{\delta}

\newcommand{\op}{\oplus}
\newcommand{\la}{\lambda}
\newcommand{\tr}{\triangleright}
\newcommand{\tl}{\triangleleft}

\newcommand{\End}{\mathrm{End}}

\newcommand{\Span}{\mathrm{Span}}

\newcommand{\rk}{\mathrm{rk}}

\newcommand{\Rm}{\mathrm{R}}

\newcommand{\ad}{\mathrm{ad}}

\newcommand{\Ga}{\Gamma}
\newcommand{\g}{\mathfrak{g}}
\renewcommand{\a}{\mathfrak{a}}
\renewcommand{\b}{\mathfrak{b}}
\renewcommand{\k}{\mathfrak{k}}
\newcommand{\h}{\mathfrak{h}}

\newcommand{\s}{\mathfrak{s}}

\newcommand{\eps}{\epsilon}

\newcommand{\nn}{\nonumber}

\renewcommand{\l}{\mathfrak{l}}
\renewcommand{\c}{\mathfrak{c}}

\newcommand{\al}{\alpha}
\newcommand{\bt}{\beta}

\newcommand{\prt}{\partial}
\newcommand{\nbl}{\nabla}

\newcommand{\be}{\begin{eqnarray}}
\newcommand{\ee}{\end{eqnarray}}

\newtheorem{thm}{Theorem}[section]
\newtheorem{propn}[thm]{Proposition}
\newtheorem{lemma}[thm]{Lemma}
\newtheorem{corollary}[thm]{Corollary}

\newtheorem{definition}[thm]{Definition}
\newtheorem{example}[thm]{Example}

\newcount\prg

\newcommand{\parag}{\advance\prg by1 {\noindent\bf\thesection.\the\prg\hspace{6pt}}}

\begin{document}
\title{Mickelsson algebras and inverse Shapovalov form}

\author{
Andrey Mudrov${}^{\dag,\ddag}$, Vladimir Stukopin${}^\ddag$
\vspace{10pt}\\
\small ${}^{\dag}$ University of Leicester, \\
\small University Road,
LE1 7RH Leicester, UK,
\vspace{10pt}\\
\small
${}^{\ddag}$ Moscow Institute of Physics and Technology,\\
\small
9 Institutskiy per., Dolgoprudny, Moscow Region,
141701, Russia,
\vspace{10pt}\\
\small
 e-mail: am405@le.ac.uk,  stukopin.va@mipt.ru
}
\date{}

\maketitle

\begin{abstract}
Let $\Ac$ be an  associative algebra  containing the classical or quantum
universal enveloping algebra $U$  of a semi-simple complex Lie algebra.
Let $\Jc\subset \Ac$ designate the left ideal generated by positive root vectors in $U$.
We construct  the reduction algebra  of the pair $(\Ac,\Jc)$   via the inverse Shapovalov form of $U$.
\end{abstract}
\begin{center}
\end{center}
\\
{\small \underline{AMS classification codes}: 17B10, 17B37.}
\newpage

\section{Introduction}
\label{SecIntro}
Step algebras were invented by J. Mickelsson as a tool for restricting representation of a simple Lie group to its semi-simple subgroup in his
studies of Harish-Chandra modules \cite{Mick}.
They provide a representation free formulation of the raising and lowering operators technique devised  by theoretical physicists for the needs
 of many-body problems and elementary particle physics \cite{NM}. Nowadays these algebras find numerous applications to harmonic analysis, \cite{Zh1,DER},
  Gelfand-Tsetlin bases \cite{Mol}, Yangians \cite{KN1,KN2,KN3} {\em etc}.

Mickelsson algebras are a  special case of reduction algebras associated with a pair $(\Ac,\Jc)$
where $\Ac$ is an associative algebra with unit and $\Jc\subset \Ac$ is a left ideal. The reduction algebra $R(\Ac,\Jc)$ is
a quotient by $\Jc$ of  its normalizer  $N(\Jc)$, that is  the maximal subalgebra in $\Ac$ where
$\Jc$ is a two-sided ideal. Elements of $R(\Ac,\Jc)$ are can be naturally viewed as homomorphisms of the $\Ac$-module $\Vc=\Ac/\Jc$ via the right multiplication.
There is a natural isomorphism between $R(\Ac,\Jc)$ and $\Vc^+=\ker \Jc\subset \Vc$ induced by the projection $\Ac\to \Vc$.

If $\Ac$ contains the universal  enveloping algebra of a semi-simple Lie algebra $\g$, one can consider the ideal $\Jc$ generated by
a maximal nilpotent subalgebra, $\g_+$ relative to a fixed triangular decomposition of $\g$.
Then  the Mickelsson algebra $Z(\Ac,\g)=R(\Ac,\Jc)$ becomes responsible
for restriction of $\Ac$-representations to $\g$. One can view it as the symmetry algebra of $\g$-singular
vectors, which parameterize $\g$-submodules, for instance, in
a module of highest weight when $\Ac=U(\a)$ is the universal enveloping algebra of a simple Lie algebra $\a\supset \g$.

One can similarly define the Mickelsson algebra  upon replacement of $U(\g)$ with the quantum group $U_q(\g)$, \cite{Kek}.
Then it is a deformation of its classical counterpart provided $q$ is not a root of unity. We retain the same notation $Z(\Ac,\g)$  for it although
we mostly deal with  quantum groups in this exposition. The classical version  can be obtained as the limit case $q\to 1$.

Let us point out  the most studied examples of Mickelsson algebras.
For finite dimensional $\g$ of classical type they include  $\a=\g\op \g$ with diagonal embedding of $\g$. Such $Z(\a,\g)$ is responsible for decomposition of
tensor product of $\g$-modules \cite{KO}.
Another example corresponds to  reductive pairs participating in  the Gelfand-Tsetlin reduction \cite{GTs1,Mol}.
With regard to infinite dimensional $\a$ one should mention Yangians and twisted Yangians \cite{KN1,KN2,KN3}.

The modern approach to Mickelsson algebras is based on the theory of extremal projectors \cite{AST}, after works of Zhelobenko \cite{Zh2,Zh}.
Extremal projector of $U(\g)$ and, respectively, its quantum counterpart allows to construct  $Z(\Ac,\g)$ from $\Ac$ and
provides  a Poincare-Birkhoff-Witt basis. However the  action of the extremal projector on $\Ac$ is hard to explicitly compute in general.
We addressed this problem in \cite{MS} utilizing an approach of Hasse diagrams associated with classical or quantum $\g$-modules.
As a result, we obtained explicit expression for a PBW basis in $Z(\a,\g)$ through a basis in $\a$ in the classical case
and matrix entries of quantum Lax operators for $q\not =1$.

The Hasse diagram technique of \cite{MS} appears  to be pretty close to the one used for  inversion of the  Shapovalov form in \cite{M1}.
This form is a well established concept in representation theory and applications \cite{Shap}. Its inverse  is a key
ingredient of dynamical quantum groups \cite{EV1} and equivariant quantization of homogeneous spaces \cite{AL,M5}.
Its matrix elements participate in a construction of Shapovalov elements delivering singular vectors in reducible Verma modules \cite{M5}.
On the other hand,  a relation between  Shapovalov elements and Mickelsson algebras  has been observed long  ago \cite{Brun, Cart}.
In this paper, we explore those connections to a full extent  and reformulate the results of \cite{MS} with the help of the inverse Shapovalov  form. More precisely, with   its
lift $\Sc\in U_q(\g_+)\tp U_q(\b_-)$ (a completed tensor product), where $\b_-\subset \g$ is the negative Borel subalgebra and $\g_+$ its opposite  maximal nilpotent Lie subalgebra in $\g$.
This finding should   not be a surprise
because inverse contravariant forms and extremal projectors are two alternative ways to construct  $\g_+$-invariants in tensor products of $\g$-modules, see e.g. \cite{M3}.

Let us briefly explain two alternative constructions of $Z(\Ac,\g)$ presented in this paper, focusing on the $q\not =1$ setting. We consider two  comultiplications $\Delta$ and $\tilde \Delta$ on $U_q(\a)$ intertwined by the R-matrix with a  cut-off diagonal Cartan factor.
Suppose that $X$ is a $U_q(\g_+)$-module that is realized as a submodule in $\Ac$ under the adjoint $U_q(\b_+)$-action restricted to $U_q(\g_+)$.
We assume that  $X$ is   $U_q(\h)$-diagonalizable  with finite dimensional weight subspaces.
Let $X^*$ denote the dual vector space to $X$.  We associate with $X$ a matrix $S_X\in \End(X)\tp \Ac$ and a vector $\psi_{X^*}\in X^*\tp \Ac$
that feature the following transformation properties:
$$
(1\tp u)\psi_{X^*}=\psi_{X^*}\Delta(u), \quad \Delta(u) S_X\in  \End(X)\tp \Jc,   \quad \forall u\in U_q(\g_+).
$$
Here we mean the natural right $\End(X)$-action on $X^*$ by duality.
Then the $\Ac$-entries of the vector  $Z_{X^*}=\psi_{X^*} S_X $ are  in $\Vc^+ $, modulo the ideal $\Jc$.

A "dual"  construction involves a matrix $\tilde S_X\in \End(X)\tp \Ac$ and a vector $\psi_X\in X\tp \Vc$ obeying
$$
(1\tp u)\tilde S_X=\tilde S_X'\tilde \Delta(u), \quad \tilde \Delta (u)\psi_X =0
$$
for each homogeneous $u\in U_q(\g_+)$ and certain $\tilde S_X'$  depending on the weight of $u$.
Then  components of the vector $Z_X=\tilde S_X \psi_X$
 are in $ \Vc^+ \simeq Z(\Ac,\g)$. Moreover, $Z_X=(\id\tp \wp)(\psi_X)$, where $\wp$ is the extremal projector of $U_q(\g)$.


The matrix $S_X$ is a reduction
to the representation $X$ of the universal Shapovalov matrix $\Sc\in U_q(\g_+)\tp U_q(\b_-)$ relative to left Verma modules and comultiplication $\Delta$.
The matrix $\tilde S_X$ is a reduction of a twisted universal Shapovalov matrix of {\em right} Verma modules corresponding to comultiplication $\tilde \Delta$.
A remarkable thing about the constructions described above is that a morphism from the category of right modules participates  in the category
of left modules and {\em vice versa}.


Let us emphasize  that the appearance of two different comultiplications in our approach is not just a reflection of the two alternative constructions
of $Z(\Ac,\g)$. The Shapovalov matrix of either approach  is itself expressed through the twist linking $\Delta$ and $\tilde \Delta$.
This interpretation of the main building block becomes obscure in the limit $q\to 1$ although the classical version of the theory can be developed in parallel.

The paper is organized in the following way.
After a short reminder of quantum group basics  in Section \ref{SecQuantGroups} we give a brief introduction into the theory of Mickelsson algebras in
Section \ref{SecMickAlg}. Section \ref{secAlgRout} is devoted to Hasse diagrams associated with representations of quantum groups and to algebra of routes.
Construction of $Z(\Ac,\g)$ via right Shapovalov matrix is presented in Section \ref{secRigthShMat}.
An alternative construction via left Shapovalov matrix is given in Section \ref{secLeftShMat}.
\section{Basic notation and quantum group conventions}
\label{SecQuantGroups}
For details on quantum groups, the reader is referred either to original Drinfeld's talk \cite{D1} or to the textbook \cite{ChP}.
In this exposition, we are also  using notation of \cite{MS}. Here we remind  the basics for reader's convenience.
Throughout the paper,  $\g$ is a semi-simple complex Lie algebra with
a triangular decomposition  $\g=\g_-\op \h\op \g_+$, where $\g_\pm$ are  maximal nilpotent Lie subalgebras and
 $\h$ the Cartan subalgebra.

Let
$\Rm \subset \h^*$ denote the root system of $\g$ and  $\Rm^+$ the subset of positive roots with basis $\Pi$
of simple roots. The root lattice generated by $\Pi$ is denoted by $\Ga\subset \h^*$, with the positive semigroup $\Ga_+=\Z_+\Pi\subset \Ga$.
Height   $\hgt(\mu)\in \Z_+$ of a weight $\mu$ from $\Ga_+$ is  the sum of its coordinates in the expansion over the basis $\Pi$.

An $\ad$-invariant Killing form on $\g$ is denoted by  $(\>.\>,\>.\>)$; it is restricted  to $\h$ and transferred to $\h^*$ by duality.
For every $\la\in \h^*$ there is   a unique element $h_\la \in \h$ such that $\mu(h_\la)=(\mu,\la)$, for all $\mu\in \h^*$.
We assume that $q\in \C$ is non-zero and  not a root of unity. For  $z\in \h+\C$, we use notation  $[z]_q=\frac{q^{z}-q^{-z }}{q-q^{-1}}$.

The standard Drinfeld-Jimbo quantum group $U_q(\g)$ is an associative $\C$-algebra with unit generated by $e_\al$, $f_\al$, and invertible $q^{h_\al}$ labeled by simple roots $\al$. In particular, they satisfy
$$
q^{h_\al}e_\bt=q^{ (\al,\bt)}e_\bt q^{ h_\al},
\quad
[e_\al,f_\bt]=\dt_{\al,\bt}[h_\al]_q,
\quad
q^{ h_\al}f_\bt=q^{-(\al,\bt)}f_\bt q^{ h_\al},\quad \al, \bt \in \Pi.
$$
The full set of relations can be found in \cite{ChP}.

We shall consider two Hopf algebra structures on $U_q(\g)$ with comultiplications
$$
\Delta(f_\al)= f_\al\tp 1+q^{-h_\al}\tp f_\al,\quad \Delta(q^{\pm h_\al})=q^{\pm h_\al}\tp q^{\pm h_\al},\quad\Delta(e_\al)= e_\al\tp q^{h_\al}+1\tp e_\al,
$$
$$
\tilde \Delta(f_\al)=f_\al\tp 1 +q^{h_\al}\tp f_\al,\quad \tilde \Delta q^{\pm h_\al}=q^{\pm h_\al}\tp q^{\pm h_\al}, \quad \tilde \Delta(e_\al)=e_\al\tp q^{-h_\al}+1\tp e_\al,
$$
defined on the generators and extended further as algebra homomorphisms. 
Their antipodes are given by
$$
\gm( f_\al)=- q^{h_\al}f_\al, \quad \gm( q^{\pm h_\al})=q^{\mp h_\al}, \quad \gm( e_\al)=- e_\al q^{-h_\al},
$$
$$
\tilde \gm( f_\al)=- q^{-h_\al}f_\al, \quad \tilde \gm( q^{\pm h_\al})=q^{\mp h_\al}, \quad \tilde \gm( e_\al)=- e_\al q^{h_\al}.
$$
The counit homomorphism $\eps\colon U_q(\g)\to \C$ is the same for $\Delta$ and $\tilde \Delta$ and  returns
$$
\eps(e_\al)=0, \quad \eps(f_\al)=0, \quad \eps(q^{h_\al})=1
$$
on the generators.

We denote by
$U_q(\h)$ the Cartan subalgebra in $U_q(\g)$ generated by $\{q^{\pm h_\al}\}_{\al\in \Pi}$.
The subalgebras generated by $\{e_\al\}_{\al\in \Pi}$ and $\{f_\al\}_{\al\in \Pi}$ are denoted as $U_q(\g_+)$, and, respectively, $U_q(\g_-)$.
Quantum Borel subgroups are defined as $U_q(\b_\pm)=U_q(\g_\pm)U_q(\h)$; they are Hopf subalgebras in $U_q(\g)$.

Let $\Ru$ denote  a quasitriangular structure relative to $\Delta$ and set $\check{\Ru}=q^{-\sum_{i}h_i\tp h_i}\Ru$, where $\{h_i\}_{i=1}^{\rk \g}$ is an orthonormal basis in $\h$.
One can choose $\Ru$ such that $\check{\Ru}\in U_q(\g_+)\tp U_q(\g_-)$ (a completed tensor product).
The element $\check{\Ru}$ is a Hopf algebra twist relating the two comultiplications:
$$
\check{\Ru}\Delta(u)=\tilde \Delta(u)\check{\Ru}, \quad \forall u\in U_q(\g).
$$
The element $\check{\Ru}$ plays a central role in this exposition.


Given a $U_q(\h)$-module $V$, a non-zero vector $v$ is said to be of weight $\mu=\wt(v)$ if $q^{h_\al}v=q^{(\mu,\al)} v$ for all $\al\in \Pi$.
The linear span of such vectors is denoted by $V[\mu]$.
A $U_q(\g)$-module $V$ is said to be of highest weight $\la$ if it is generated by a weight vector $v\in V[\la]$ that
is killed by all $e_\al$. Such a vector $v$ is called highest; it is defined up to a non-zero scalar multiplier.

For each $\mu\in \h^*$ we denote by  $\tau_\mu$ the automorphism of the algebra $U_q(\h)$ defined by the assignment $\tau_\mu\colon  q^{\pm h_\bt}\mapsto q^{\pm h_\bt}q^{\pm(\mu,\bt)}$.
We extend it to an automorphism of  $U_q(\b_-)$ such that $\tau_\mu \colon  h \phi\mapsto (\tau_{\mu} h) \phi$ for $h\in \hat U_q(\h)$ and $\phi\in  U_q(\g_-)$.
It is an algebra map indeed because
  $\phi h= \tau_{-\nu}(h)\phi$ for $\phi$ of weight $\nu$ (with respect to the adjoint action) and
$\tau_\mu\circ \tau_\nu=\tau_\nu\circ \tau_\mu$ for all $\mu, \nu\in \h^*$.

\section{Mickelsson algebras}
\label{SecMickAlg}
In this section we recall the definition and  basic facts about Mickelsson algebras.

Let $U$ denote  either the classical or quantum universal enveloping algebra of $\g$ with
triangular factorization $U^-U^0U^+$ relative to the polarization $\g=\g_-\op \h\op \g_+$. The subalgebras $U^\pm$ are endowed with natural grading by setting
$\deg(u)=1$ on their simple root generators. $u$. By $B^\pm$ we denote the quantum Borel subalgebras $U^\pm U^0$.

Suppose that $U$ is contained in a unital associative algebra $\Ac$, and $\Ac$ is diagonalizable over $U^0$ with finite dimensional weight spaces.
Define a left ideal $\Jc=\Jc_+=\Ac\g_+$ as the one generated by $e_\al$ with $\al \in \Pi$.
Similarly one defines a right ideal $\Jc_-=\g_- \Ac$ generated by $f_\al$ with $\al \in \Pi$.

Let  $N(\Jc)$ denote the normalizer of $\Jc$ in $\Ac$, i.e.
the set of elements  $a\in \Ac$ such that $\Jc a \subset \Jc$. It is clear that $N(\Jc)$ is an algebra and $\Jc$ is a two-sided ideal in it.

\begin{definition}
  The quotient $N(\Jc)/\Jc$ is called Mickelsson (step, reduction) algebra of the pair $\Ac\supset U$ and denoted by $Z(\Ac,\g)$.
\end{definition}
In particular,  $\Ac$ can be  either classical or quantized universal enveloping algebra
of a simple Lie algebra $\a$ containing $\g$ (provided the embedding $\a\supset \g$  is quantizable in the $q\not=1$ case). Then we denote $Z(\Ac,\g)$ by $Z(\a,\g)$.

We denote by  $\hat U^0$ the ring of fractions over $[h_\mu-c]_q$ ($h_\mu-c$ in the classical case)
where $\mu \in \Gamma_+$ and $c$ ranges in a subset in $\C$ depending on the weight structure of $\Ac$. In particular,
if all weights of $\Ac$ are integral, it is sufficient to require  $c\in \Qbb$.
Assuming the adjoint action of $U^0$ on $\Ac$ diagonalizable,
 we extend them
accordingly, to $\hat \Ac$ and $\hat \Jc_\pm$. The Mickelsson algebra of the pair $(\hat \Ac, \hat \Jc)$ will be denoted by $\hat Z(\Ac, \Jc)$.
It is clear that $\hat Z(\Ac, \Jc)\simeq Z(\Ac, \Jc)\tp_{U^0}\hat U^0$.

We suppose that there exist further extensions  $\hat {U}\subset \breve{\U}$, $\hat {\Ac}\subset \breve{\Ac}$,
$\hat \Jc_\pm\subset \breve{ \Jc}_\pm$
such that $\Jc_\pm= \breve{ \Jc}_\pm\cap \Ac$, and there is an element (extremal projector) $\wp\in \breve  { \U}$ of zero weight satisfying
\be
\Jc_+\wp =0=\wp\Jc_-, \quad \wp^2=\wp, \quad \wp-1\in \breve { \Jc}_-\cap \breve  {\Jc}_+.
\label{extr_proj}
\ee
Such an extension does exist in the trivial case $\Ac=U$; then it requires a completion of $\hat \U$ with certain infinite series.
In general, it is sufficient to assume, for instance, that  $\Ac$ is a free left or right regular $U$-module generated by a locally finite adjoint $U$-submodule.
In particular, if $\Ac=U(\a)$ of a finite dimensional Lie algebra $\a\supset \g$  or $\Ac=U_q(\a)\supset U_q(\g)=U$, where $\g$ is the  semi-simple part of a Levi subalgebra $\g\subset \a$.

The exact formulas for $\wp$ can be found in $\cite{K}$. Here we note  that $\wp$ is a product of factors
each of which is a  series in $f_\al^k e_\al^k$, $\al \in \Rm^+$,  with coefficients in $\hat U^0$.

The inclusion in (\ref{extr_proj}) implies the following isomorphisms of $\hat U^0$-modules:
$$
\hat \Ac/\hat \Jc_+ \simeq \hat \Ac\wp, \quad  \hat \Ac/\hat \Jc_-\simeq  \wp \hat \Ac, \quad
\hat \Ac/(\hat \Jc_-+\hat\Jc_+) \simeq  \wp \hat \Ac\wp.
$$
We denote by $\hat \Vc$ $\Ac$-module $\hat \Ac/\hat \Jc_+$  and by $1_{\hat\Vc}$ the image of $1$ in $\hat \Vc$.

\begin{thm}[cf. \cite{K}]
\label{thmMick}
Projection $\varpi\colon  \hat \Ac\to \breve {\Ac}$,
$
a\mapsto \wp a\wp,
$ factors through an isomorphism of algebras
$$
\hat Z(\Ac,\g)\to \wp \hat \Ac \wp.
$$
\end{thm}
From now on we suppose that $\Ac$ is a free $U^-\tp B^+$-bimodule relative to the regular actions
generated by (a basis of) a vector space $\Zc\subset \Ac$.
Then $\Zc \hat U^0$ is isomorphically mapped onto the double coset algebra $\hat \Ac/(\hat \Jc_- + \hat \Jc_+)$.
Two types of $\Zc$ may be of interest within our approach:
a) either $\Zc$ is  a  locally
finite adjoint  $B^+$-submodule or
b)
$\Zc$ has a basis of ordered
monomials $\psi_1^{m_1}\ldots \psi_k^{m_k}$
generated by weight elements $\{\psi_i\}_{i=1}^k$ of a finite dimensional
  $B^+$-submodule $\Zg$.
In the latter case we say that $\Zg\subset \Zc$ delivers a Poincare-Birkhoff-Witt (PBW) basis over $U$.
The set $\{\psi_i\}_{i=1}^k$ is called a PBW system and
its elements PBW generators.
It is known that the image of a PBW-system in $\hat \Ac$
generates a PBW-basis in $\pi \hat \Ac\pi\simeq \hat Z(\Ac,\g)$, as a $\hat U^0$-module.

More generally, we may assume that $\Zc$ and $\Zg$ satisfy the above requirements modulo $\Jc$, that is, upon  embedding in $\Vc$.

In particular, we are interested in the following two cases.
First suppose that  $\Ac=U(\a)$ is the universal enveloping algebra of a Lie algebra  $\a\supset \g$. Let
$\Zg\subset \a$ be a $\g$-module complementary  $\g$.
Any basis in $\Zg$ is a PBW system over $U$ according to the PBW theorem.

The second case of our concern is when $\a$ is a simple Lie algebra and $\g$ is the derived Lie algebra of a Levi subalgebra in $\a$.
Then we take $\Ac=U_q(\a)$, $U=U_q(\g)$,  and
consider the $U_q(\a)$-module $\a_q$ deforming the adjoint module $\a$.
It contains a $U_q(\g)$-submodule $\g_q$ deforming the adjoint module $\g$.
We realize the quotient $\a_q/\g_q$ as a $U_q(\g)$-submodule in $\Ac$ (possibly modulo $\Jc$) and take it for the role of $\Zg$.
There is a basis in $\Zg$ delivering  a PBW system in $U_q(\a)$ over $U_q(\g)$ upon extension of the ring of scalars to $\C[[\hbar]]$ with $q=e^\hbar$,
see Section \ref{secLeftShMat}.

Other examples may include the Heisenberg double $\Ac=U^*\rtimes U$, where $U^*$ is the function algebra on the (quantum) algebraic group of $\g$,
$\Ac =U\tp U$ with the diagonal embedding of $U=U_q(\g)$, see Section \ref{SecLeftZ},  $\Ac=\mathrm{T}(\Zg)\rtimes U$, where $\mathrm{T}(\Zg)$ is the tensor algebra of a $U$-module $\Zg$ {\em etc}.
\section{Algebra of routes}
\label{secAlgRout}
In this section we briefly recall an algebraic structure on routes in  Hasse diagram associated with representations of $U_q(\g)$, following \cite{MS,M1}; more  details can be found therein.

\subsection{Hasse diagrams}
Let $X$ be a $U^+$-module that is extendable to a module over $B^+$ and diagonalizable over $U^0$ with $\dim X[\mu]<\infty $ for all $\mu\in \h^*$.
Such $U^+$-modules we call graded. Note that they may have non-isomorphic $B^+$-extensions. 
We assume that for any pair of vectors from $X$, their weight difference is in the root lattice $\Ga$.
We associate with $X$ a Hasse diagram $\Hg(X)$ as follows. It is an oriented graph whose nodes are elements of a fixed weight basis $\{v_i\}_{i\in I_X}\subset X$,
$\wt(v_i)=\nu_i$. We will identify them with elements of the index set $I_X$.
 Arrows $i\stackrel{e_\al}{\longleftarrow} j$ are marked with simple  root vectors $e_\al$ if $\pi^\al_{ij}=\pi(e_\al)_{ij}\not =0$,
then  $\nu_i-\nu_j=\al$.

 A sequence of adjacent arrows
$$
m_1
\stackrel{\al_1}{\longleftarrow} m_2\stackrel{\al_2}{\longleftarrow} \ldots \stackrel{\al_k}{\longleftarrow} m_{k+1}
$$
is called path (of length $k$) from $m_{k+1}$ to $m_1$.
A node $i$ is superior to a  node $j$, $i\succ j$, if there is a path from $j$ to $i$. This defines a partial order on $\Hg(X)$.

A route $\vec m =(m_1,\ldots m_{k+1})$ from $j$ to $i$ is an arbitrary ordered sequence
$$
i=m_1\succ \ldots \succ m_{k+1}=j.
$$
Typically we orient routes in the right-to-left ascending order.
We denote $\max(\vec m)=i $ and $\min(\vec m)=j$ and say that  $\vec m$ is a route $i\dashleftarrow j$ or write it as  $i\stackrel{\>\>\vec m}{\dashleftarrow} j$.
We will also suppress one of the nodes to emphasize the start or end node: for instance, $i\stackrel{\>\>\vec m}{\dashleftarrow} $ means that $\vec m$ is route $i\dashleftarrow$
(terminating at $i$) whereas $\stackrel{\>\>\vec m}{\dashleftarrow} j$ is a route $\vec m$ starting from $j$.

 The number  spans in $\vec m$ is called length of $\vec m$ and denoted by $|\vec m|$.
For instance, a path to $i$ from  $j$ is a route $i\dashleftarrow j$ of maximal length equal to $\hgt(\nu_i-\nu_j)$ that is the number of arrows in it.


Given $\stackrel{\>\>\vec m}{\dashleftarrow} j \succ i\stackrel{\>\>\vec n}{\dashleftarrow} $, we write $\vec m\succ \vec n$. Then there is a route $(\vec m,\vec n)$ that includes all nodes from $\vec m$ and $\vec n$.
We drop the vector superscript for routes consisting of one node.
For instance, $(\vec m, \vec n)=(i,j,k)$ if $\vec m=(i)\succ (j,k)=\vec n$.

For two routes $i\stackrel{\>\>\vec m}{\dashleftarrow} k\stackrel{\>\>\vec n}{\dashleftarrow} j$ we get
a route $i\stackrel{\>\>\vec m\cdot \vec n}{\dashleftarrow} j$, the concatenation of   $\vec m$ and $\vec n$.
Its set of nodes is the union of the two. Concatenation is a partial associative operation on routes.

Removing an arbitrary subset of nodes from a route is a route again (possibly empty).

\subsection{Auxiliary module $\Phi_X$}

Denote by  $\Phi_X$ a free right  $\hat U^0$-module generated by routes in $\Hg(X)$. Introduce a left  $\hat U^0$-action by
assigning a weight of $\nu_j-\nu_i$ to a route $i\stackrel{\>\>\vec m}{\dashleftarrow} j$: it is regarded as a character $\hat U^0\to \C$ relative to the induced adjoint action
on $\Phi_X$.
Extend concatenation  as a partial operation
to $\hat U^0$-lines  $\vec m \hat U^0$, by associativity.

Define   $\Fc=\frac{1}{q-q^{-1}}(\check{\Ru}-1\tp 1)\in U^+\tp U^-$
and consider a matrix $F=(\pi\tp \id)(\Fc)=\sum_{i,j\in I_X}e_{ij}\tp \phi_{ij}$, where $e_{ij}$ are the matrix units obeying
$e_{ij}e_{mn}=\dt_{jm}e_{in}$. The entry $\phi_{ij}$ carries weight $\nu_j-\nu_i$, so the matrix $F$ is strictly lower triangular.

To each route $\vec m$ from $\Hg(X)$ we assign an element
$$
 \phi_{\vec m}= \phi_{m_1,m_2} \ldots  \phi_{m_{k-1},m_k}\in U^-
$$
end extend this assignment to a $\hat U^0$-bimodule homomorphism $p_\Phi\colon \Phi_X\to \hat B^-$.
For a route $(i)$ of zero length we set $\phi_i=1$.
The map $p_\Phi$ is multiplicative with respect to concatenation:
$$
p_\Phi( \vec m\cdot \vec n)=p_\Phi( \vec  m)p_\Phi( \vec n), \quad \vec  m,\vec  n\in \Phi_X.
$$

Arrows in $\Hg(X)$ are in bijection with ordered pairs $(l,r)$  of nodes they connect. We call such pairs simple.
Set $P(\al)=\{(l,r)\in I_X^2|\>l\stackrel{\al}{\longleftarrow}r\}$ for $\al\in \Pi$.
We define an operator $\prt_{l,r}\colon \Phi_X\to \Phi_X\tp_{\hat U^0} \Phi_X$ for each $(l,r)\in P(\al)$ as follows.
We set it zero on routes of zero length.
On routes of length $1$ we put
\be
\begin{array}{rrccc}
\prt_{l,r} (l,r)&=&(l)\tp (r)[h_\al]_q  ,\\
\prt_{l,r} (l,j)&=&- (l)\tp q^{-h_\al}(r,j),
&r\succ j,
\\
\prt_{l,r}(i,r) &=&(i,l)q^{h_\al}\tp (r),
&i\succ l,
\end{array}
\label{prt on phi}
\ee
and zero otherwise.
We extend it to all routes as  a homomorphism of right $\hat U^0$-modules and a derivation with respect to concatenation:
$$
\prt_{l,r}(\vec m\cdot \vec n)=(\prt_{l,r} \vec m)\cdot \bigl((l)\tp  \vec n\bigr) + \bigl(\vec m\tp (r)\bigr)\cdot (\prt_{l,r} \vec n).
$$
At most one concatenation factor survives the action of $\prt_{l,r}$; then $\cdot$ in the right-hand side makes sense.

Define $p_{\Phi\Phi}\colon \Phi_X\tp_{ \hat U^0}\Phi_X \to \hat B^-$ as  the composition of
$p_\Phi\tp p_\Phi$ with the multiplication on $\hat B^-$. It is also a $\hat U^0$-bimodule homomorphism.
\begin{propn}
\label{intertwing loc}
For each $\al\in \Pi$ and each $\xi\in  \Phi_X$,
\be
\label{e-action}
e_\al p_\Phi (\xi)=p_{\Phi\Phi}\circ \sum_{(l,r)\in P(\al)}\pi^\al_{lr}\prt_{l,r}(\xi)+\bigl(\tau_{\al}^{-1}p_\Phi(\xi)\bigr) e_\al.
\ee
\end{propn}
\begin{proof}
The proof is analogous to the proof of  \cite{MS}, Proposition 6.4  with the difference that  the calculations are done exactly rather than
modulo $\Jc$. That accounts for the appearance of the right term on the right-hand side of (\ref{e-action}).

For $\xi=(i,j)$ we have $p_\Phi(\xi)=\phi_{ij}$.
Comparing the formula
\be
e_{\al}\phi_{ij}- \phi_{ij} e_{\al}=  \sum_{k\in I_X}\phi_{ik}q^{h_{\al}} \pi^\al_{kj}  -\sum_{k\in I_X} \pi^\al_{ik} q^{-h_{\al}}\phi_{kj}+\pi^\al_{ij}[h_\al]_q,
\quad
\al\in \Pi,
\label{intertwiner_F}
\ee
with (\ref{prt on phi}) we prove (\ref{e-action})  for all routes of length 1. Note that $\tau_\al^{-1}$ is identical on $U^-$
so the first term in right-hand side of (\ref{e-action}) gives an expression for the commutator with $e_\al$.
Using the Leibnitz rule we extend it to all routes, which generate $\Phi_X$ as a right $\hat U^0$-module.
To complete the proof, we utilize the relation $e_\al f h - \tau^{-1}_\al(fh)e_\al=[e_\al, f]h$ for all $f\in U^- $ and $h\in \hat U^0$.
\end{proof}
\noindent
Proposition \ref{intertwing loc} is the rationale for introducing the auxiliary module $\Phi_X$, as it allows to reduce the adjoint $U^-$-action
on $B^-$ to the action of simpler operators $\prt_{l,r}$.
\subsection{A construction of $\hat Z(\Ac,\g)$}
\label{SecLeftZ}
For a $U^+$-module $X$ suppose that  $\hat\Ac/\hat \Jc= \hat \Vc$ contains a submodule that is isomorphic to its left dual $X^*$. 
Then we can construct a $U^-$-invariant tensor $\psi_X=\sum_{i\in I_X} x_i\tp \psi_i\in X\tp X^*$. We call it right Mickelsson generator
relative to $X$.

In the previous section, we defined $\phi_{\vec m}$ as a $\B^-$-valued function of routes in $\Hg(X)$.
We extend the assignment $I_X\ni i\mapsto \psi_i\in \Vc$ to a  function on routes with values in a $B^-$-module $\Vc$ by
setting   $\psi_{\vec m}=\phi_{\vec m}\psi_{m_k}\in \Vc$, for  $\vec m=(m_1,\ldots, m_k)$.

Denote by $\rho\in \h^*$ the half-sum of positive roots. For each weight $\mu\in \Gamma_+$ define
\be
\eta_{\mu}=h_\mu+(\mu, \rho)-\frac{1}{2}(\mu,\mu), \quad
\tilde \eta_{\mu}=h_\mu+(\rho,\nu)+\frac{1}{2}(\mu,\mu)
\ee
as elements of $\h+\C$. Then $q^{\eta_\mu}$ and $q^{\tilde\eta_\mu}$ are elements of $\hat U^0$.
Clearly $\eta_\mu=\tilde \eta_\mu-||\mu||^2$. Denote a rational trigonometric function
$$
z\mapsto \varphi(z)=\frac{q^{-z}}{[z]_q},
$$ where $z$ is an indeterminate.
Next we  introduce a system of elements of $\hat U^0$ as functions of routes. They will be used as left and right multipliers in  construction of $\hat Z(\Ac,\g)$.

For each node $i\in I_X$ we denote $\eta_i=\eta_{\nu_i}$. For each ordered pair of nodes $i\succ j$ we define $\eta_{ij}=\eta_{\nu_i-\nu_j}$, $\tilde \eta_{ij}=\tilde \eta_{\nu_i-\nu_j}$, and set
$$
A^j_i=\varphi(-\eta_{ij}),\quad \tilde A^i_j=\varphi(\tilde \eta_{ij}),\quad B^i_j=\varphi(\eta_{i}-\eta_j).
$$
Furthermore, for each route $\vec m=(m_1,\ldots, m_k)\not =\varnothing $ we define
$$
A^j_{\vec m}=A^j_{m_1}\ldots A^j_{m_k}, \quad \tilde A^i_{\vec m}=\tilde A^i_{m_1}\ldots \tilde A^j_{m_k}, \quad B^i_{\vec m}=B^i_{m_1}\ldots B^i_{m_k},
$$
assuming $i\succ \vec m$ and $\vec m\succ j$.
One can check that $\tau_{\nu_i}(\eta_i-\eta_{j})=\tilde \eta_{ij}$ for all $i\succ j$. This gives a relation
\be
\tilde A^i_{\vec m} =\tau_{\nu_i}(B^i_{\vec m}).
\label{A-B}
\ee
It is also convenient to set $A^i_i=\tilde A^i_i=B^i_i=1$ for all $i\in I_X$.

\begin{propn}[\cite{MS}]
For each $i\in I_X$, the element
\be
z_i=\psi_{i}+\sum_{i\succ  \vec m \not =\varnothing }\psi_{(i,\vec m)}B^i_{\vec m}=\wp \psi_i
\label{norm_element}
\ee
belongs to $\in \hat\Vc^+$.
With identification  $\hat \Vc^+\simeq \hat Z(\Ac,\g)$, $z_i$ is  an element of $\hat Z(\Ac,\g)$.
  \end{propn}
\noindent
Picking appropriate $X$ such that $X^*\subset \Ac$ one can obtain the entire $\hat Z(\Ac,\g)$. Alternatively, one can use $X$ to construct a PBW basis, as in \cite{MS}.
Next we give an example that did not enter \cite{MS}. It is motivated by a realization of Drinfeld's quantum double of $U$ \cite{D1} as a subalgebra in $U\tp U$ \cite{RS}.
\begin{example}
  \em
  Let $U$ be the quantum group $U_q(\g)$ of a simple Lie algebra $\g$ and set $\Ac=U_q(\g)\tp U_q(\g)$ with the diagonal embedding of $U$ via $\Delta$.
The corresponding Mickelsson algebra is responsible for decomposing tensor product of $U_q(\g)$-modules.
Observe that an extension $\breve \Ac$ accommodating the extremal projector of $U$ does exist. The algebra $\breve U$ is spanned by series in products  $u_-u_+ h$ of same weight, where $u_\pm\in U^\pm$ and $h\in \hat U^0$. It is sufficient to take $\breve \Ac=\breve U\tp \breve U$.

In order to construct a (right) Mickelsson generator, let $(Y,\varrho)$ be a fundamental $U$-module of minimal dimension and
define an action  $\pi$ on $E=\End(Y)$ by  $\pi(u)(x)=\varrho(u^{(2)})x\varrho\bigl(\gm^{-1}(u^{(1)})\bigr)$ (in the Sweedler notation for the coproduct), for $u\in U$ and $x\in E$.
We restrict $\pi$ to $U^+$ and take
 the matrix $ \psi_E=R_{12}^{-1}R_{31}\in E\tp U\tp U$  for a Mickelsson generator. It satisfies
$$
\bigl(\varrho(u^{(1)}) \tp u^{(2)}\tp u^{(3)}\bigr)\psi_E=\psi_E\bigl(\varrho(u^{(3)}) \tp u^{(1)}\tp u^{(2)}\bigr)=\psi_E\bigl(\varrho(u )\tp 1\tp 1\bigl) \mod E\tp \hat \Jc, \quad \forall u\in U^+.
$$
This is equivalent to $U^+$-invariance of  $\psi_E$ as a tensor from $E\tp \hat \Vc$.

The module $E$ is completely reducible and contains a submodule $X\subset E$ quantizing the classical coadjoint module $\g^*\simeq \g$
complementary  to  diagonally embedded $\g\subset \g\op \g$ (the classical double).
The invariant  projection  $E\to X$
  takes $\psi_E$ to $\psi_X$ whose components generate   a PBW basis in $\Ac$ over $U$, upon extension of the ring of scalars to formal power series
  in $\hbar=\ln q$.
Respectively, the components of the vector $Z_{X}=\tilde S_{X}\psi_{X}$ generate a PBW basis in $\hat Z(\Ac,U)$.
\end{example}

\section{Right Shapovalov matrix and  $\hat Z(\Ac,\g)$ }
\label{secRigthShMat}
In this section we relate the construction of  $\hat Z(\Ac,\g)$ via Hasse diagrams worked out in \cite{MS}  with the inverse Shapovalov form on  $U$.

Set $q=e^\hbar$ and denote by $U_\hbar(\g)$  the $\C[[\hbar]]$-extension of $U$ completed in the $\hbar$-adic topology \cite{D1}.
Denote
$
d= \frac{1}{2}\sum_i h_i h_i + h_\rho \in U_\hbar(\h).
$
\begin{lemma}
\label{[d,.]}
  Let $\phi \in U^-$ be an element of weight $-\mu<0$.
  Then
$$
[d,\phi]=-\tilde \eta_\mu \phi=-\phi \eta_\mu.
$$
\end{lemma}
\begin{proof}
Straightforward.
\end{proof}
\noindent It follows  that the operator $q^{ [d, - ]}$
leaves  a $\C[q,q^{-1}]$-submodule $\sum_{\mu>0}\hat B^-[-\mu]\subset \hat U_\hbar(\b_-)$ invariant, and $\frac{q^{2[d, - ]}-\id}{q-q^{-1}}$
is invertible on it.
The inverse $\varphi([d, -])$  acts as right multiplication by $\varphi(-\eta_\mu )$ and left multiplication by $\varphi(-\tilde \eta_\mu )$
on each subspace of weight $-\mu<0$.
Then the operators  $\varphi(\pm D)$ with $D=\id \tp [d, -]$ are  well defined
on the subspace $\sum_{\mu>0} U^+[\mu]\tp \hat B^-[-\mu]$.

Introduce an element $\tilde \Sc$ in the completed tensor product $U^+ \tp \hat B^-$ as
\be
\label{right S}
\tilde \Sc=\sum_{n=0}^{\infty } \tilde \Sc^{(n)}, \quad \mbox{where}\quad
\tilde \Sc^{(0)}=1\tp 1,\quad \tilde \Sc^{(n+1)}= \varphi(-D)(\tilde \Sc^{(n)}\Fc),
\quad n\geqslant 0.
\ee
The series is truncated when the left tensor leg is sent to $\End(X)$, with $\dim(X)<\infty$. If we define $\hgt(X)$ as the height of the difference
between the highest and lowest weight of $X$, then  $\tilde \Sc^{(n)}=0$ for $n>\hgt(X)$,
so the sum contains at most $\hgt(X)+1$ terms.

Formula (\ref{right S})  gives an explicit expression for $\tilde S^{(n)}_X=(\pi\tp \id)(\tilde \Sc^{(n)})\in \End(X)\tp \hat B^-$
and hence for $\tilde S_X=\sum_{n=0}^\infty \tilde S^{(n)}$ if one knows the matrix $(\pi\tp \id)(\check{\Ru})$.
It is greatly simplified in the classical limit because $\lim_{q\to 1}\Fc= \sum_{\al\in \Rm^+}e_\al\tp f_\al$, the inverse
invariant form $\g_-\tp \g_+\to \C$.
\begin{propn}
   Components  of vector $z_X=\tilde S_X \psi_X$ equal (\ref{norm_element}).
\label{Mick_vector}
\end{propn}
\begin{proof}
Let $(i,\vec m)$ with $\vec m=(m_1,\ldots, m_k)$ be a route.
Since the weight of  $\psi_{i,\vec m}$ is $-\nu_i$, formula (\ref{A-B}) implies
$
\psi_{i,\vec m}B^i_{\vec m}=\tilde A^i_{\vec m}\psi_{i,\vec m}=\tilde A^i_{\vec m} \phi_{i,\vec m}\psi_{m_k}$.
But the sum of $\tilde A^i_{\vec m}\phi_{i,\vec m}$ over  all $\stackrel{\vec m}{\longleftarrow} j$ such that  $i\succ \vec m$ and $n=|\vec m|+1$ is exactly $\tilde S^{(n)}_{ij}$, thanks to  Lemma \ref{[d,.]}.
\end{proof}

Our next objective is to demonstrate that $\tilde \Sc$ is almost  the  right Shapovalov matrix.
We lift $\tilde A^i_{\vec m}\phi_{i,\vec m}$ to
$
[i,\vec m]=\tilde A^i_{\vec m}(i,\vec m)\in \Phi_X
$
in order to use Proposition \ref{intertwing loc} and reduce the study of $e_\al$-action to the action of operators $\prt_{l,r}$
with $(l,r)\in P(\al)$.

Fix an ordered pair of nodes $i\succ j$. For   each $\prt_{l,r}$ we define 1-, 2-, and 3-chains in as $\hat U^0$-linear combinations of routes
$i\dashleftarrow j$:
\begin{itemize}
  \item
$
[\vec m]
$
if $\vec m\cap (l,r)=\varnothing$ or $\vec m\cap (l,r)\not=\varnothing$ but $\vec m\cup (l,r)$ is not a route $i\dashleftarrow j$,
  \item
$
[l,\vec \rho]+[l,r,\vec \rho], \quad \mbox{with}\quad l=i,\quad
\quad [\vec \ell,l,r]+[\vec \ell,r],\quad \mbox{with}\quad  r=j,
$
  \item
$
[\vec \ell,l,\vec \rho]+[\vec \ell,l,r,\vec \rho]+[\vec \ell,r,\vec \rho].
$
\end{itemize}
Here we assume that $i=\max(\vec \ell)$ and $j=\min(\vec\rho)$, so neither $\vec \ell$ nor $\vec \rho$ are empty.
One can prove that every route $i\dashleftarrow j$ participates in exactly one chain, cf. \cite{MS}, Lemma 7.1.

\begin{propn}
\label{chain-killer}
  The operator $\prt_{l,r}$ annihilates  1-, 3-, and left 2-chains.
\end{propn}
\begin{proof}
Let us focus on chains of routes $i\dashleftarrow j$.
First of all observe that $\prt_{l,r}$ kills all 1-chains. This case includes routes $\vec m$ whose smallest node $j$ equals $l$,
because otherwise there is a route $(\vec m,r)$ which is not  $\dashleftarrow j$, see  the definition of 1-chains above.

For 2- and 3-chains, we reduce the proof to  \cite{MS}, Lemma 7.2  by sending $\Phi_X$  to another auxiliary $\hat U^0$-module $\Psi_X$ and replacing $\prt_{l,r}$ with an operator $\nbl_{l,r}\colon \Phi_X\to \Phi_X\tp_{\hat U^0}\Psi_X$. Indeed, there is a natural  embedding $\gimel\colon\Phi_X\to \Psi_X$ that is a lift of the assignment
$\phi_{\vec m}h\mapsto \psi_{\vec m}\tau_{\nu_j}^{-1}(h)$ for all $\vec m$ with $j=\min(\vec m)$ and $h\in \hat U^0$.
It is a left $\hat U^0$-module homomorphism, but a "dynamically twisted" homomorphism of right $\hat U^0$-modules.

The map $\gimel$ links  $\prt_{l,r}$   with $\nbl_{l,r}$ by the formula
$$
\nbl_{l,r}\circ\gimel (\vec m h)=(\id\tp \gimel) \circ \prt_{l,r}(\vec m h)+\bigl(\vec m\tp \id\bigr) \cdot \nbl_{l,r}(j)\tau_{\nu_j}^{-1}(h),
$$
for all $i\stackrel{\vec m}{\longleftarrow} j$ and $h\in \hat U^0$.
In particular,  $\prt_{l,r}$ and $\nbl_{l,r}$ are intertwined on the $\hat U^0$-submodule generated by routes whose smallest node $j$ is annihilated by $\nbl_{l,r}$.
That is, exactly when  $j\not =l$, cf. \cite{MS}.
All 3- and left 2-chains are in that submodule and go over to chains $\Psi_X$, which are killed by $\nbl_{l,r}$, so the statement follows.
\end{proof}
Note with care that a right 2-chain is not killed by $\prt_{l,r}$. It is not mapped to a chain in $\Psi_X$, so the above reasoning is unapplicable in this case.
This is not incidental and  facilitates the following quasi-invariance of $\tilde \Sc$.

\begin{propn}
\label{Prop-quasi-inv}
  The matrix $\tilde \Sc\in U^+\tp \hat B^-$ satisfies the identity
\be
(1\tp e_\al)\tilde \Sc=(\id\tp \tau_\al^{-1})(\tilde \Sc)\tilde \Delta(e_\al),\quad \forall \al \in \Pi.
\label{quasi-inv}
\ee
\end{propn}
\begin{proof}
For each $i\succ \vec m\succ j$ we introduce $C^i_{\vec m,j}=\tau_{\nu_j}(B^i_{\vec m,j})\in \hat U^0$, so that
$$
[i,\vec m,j]=(i,\vec m,j)C^i_{\vec m,j}\in \Phi_X.
$$
One can check that, in particular, $C^i_j=\varphi(\eta_{ij})$.

We write down the commutation relation  between $e_\al$ and the matrix entry $S_{ij}$  as
\be
e_\al \tilde S_{ij}=p_{\Phi\Phi}\sum_{i\succ  \vec m\succ j}\sum_{(l,r)\in P(\al)}\pi^\al_{lr}\prt_{l,r}
[i,\vec m,j]+\tau_\al^{-1} \tilde S_{ij}e_\al,
\label{quasi_invariant}
\ee
thanks to Proposition \ref{intertwing loc}.
Let us change the order of summations and examine the terms $$\sum_{i\succ  \vec m\succ j} \prt_{l,r}
[i,\vec m,j]=
\prt_{l,r}\sum_{i\succ  \vec m\succ j}
(i,\vec m,j)C^i_{\vec m,j}
$$
 separately for each pair $(l,r)\in P(\al)$.
This summation
can be rearranged   over $(l,r)$-chains, of which only
right 2-chains survive the action of  $\prt_{l,r}$, by Proposition \ref{chain-killer}.
For such a chain, $\prt_{lr}\bigl ([i,\vec \ell, l,r]+[i,\vec \ell, r]\bigr)$ with $r=j$, we get
\be
\label{essent}
\bigl(\prt_{l,r}(i,\vec \ell,l,r)C^i_{l}+ \prt_{l,r}(i,\vec \ell,r)\bigr)C^i_{r}C^i_{\vec \ell}=
 \bigl(\prt_{l,r}(i,\vec \ell,l,r)+ \prt_{l,r}(i,\vec \ell,r)(C^i_l)^{-1}\bigr)C^i_{r}C^i_{\vec \ell,l}.
\ee
 Explicitly, $C^i_l=\varphi(c)$, where $c=\eta_{il}-(\al,\nu_i-\nu_l)$ and $C^i_r=\varphi(\eta_{ir})$.
Notice also  that
\be
c+h_\al&=&
h_{i}-h_l+h_\al+ (\nu_i-\nu_l,\rho)-\frac{1}{2}(\nu_i-\nu_l)^2+h_\al-(\al,\nu_i-\nu_l)
=
\eta_{ir}.
\nn
\ee
Using these equalities we write down the factor before $C^i_{\vec \ell,l}$ in (\ref{essent}) explicitly:
$$
\bigl(i,\vec \ell, l)\tp (r)\bigr)\bigl([h_\al]_q+ q^{h_\al}\varphi(c )^{-1}\bigr)C^i_r =\bigl(i,\vec \ell, l)\tp (r)\bigr) \frac{(q^{h_\al}-q^{-h_\al})+ q^{h_\al}(q^{2c}-1)}{q-q^{-1}} \varphi(\eta_{ir}).
$$
The Cartan factor on the right amounts to $q^{-h_\al}$.
Thus (\ref{essent}) equals $[i,\vec\ell,l]\tp (r)q^{-h_\al}$ because zero length routes in $\Phi_X$ carry zero weights.
Returning from $\Phi_X\tp_{\hat U^0} \Phi_X$ back to $\hat B^-$ via $p_{\Phi\Phi}$ we obtain
\be
\label{quasi-invariance}
e_\al \tilde S_{ij}=\sum_{l} \sum_{i\succ  \vec \ell  \succ l}\tau_\al^{-1}\tilde A^i_{i,\vec \ell,l}\phi_{(i,\vec \ell, l)}\pi^\al_{lj}q^{-h_\al}+\tau_\al^{-1} \tilde S_{ij}e_\al
=
\sum_{l} \tau_\al^{-1}\tilde S_{il}\pi^\al_{lj}q^{-h_\al}+\tau_\al^{-1} \tilde S_{ij}e_\al.
\ee
This is a coordinate form of equation (\ref{quasi-inv}) in the representation $(X,\pi)$.
\end{proof}

Recall that a (left) Verma module $V_\la$ of highest weight $\la\in \h^*$ is induced  from
a one dimensional $B^+$-module $\C_\la$  of weight $\la$ that is trivial on $U^+\subset B^+$.
Let $v_\la\in V_\la$ denote the highest weight generator and $X$ be a finite dimensional $U$-module.
 The universal Shapovalov matrix $\Sc$ is a unique element of a completed tensor product $U^+\tp \hat B^-$ that sends $X\tp v_\la$ onto
the subspace of $U^+$-invariants in $X\tp V_\la$, for every $X$ and generic $\la$.  The matrix $\Sc$ delivers the inverse invariant
pairing between $V_\la$ and the Verma module of lowest weight $-\la$. This pairing
is equivalent to the canonical contravariant form on $V_\la$, which is a specialization of the Shapovalov form $U^-\tp U^-\to U^0$ at $\la$.

Similarly one can consider  right Verma modules $\C_\la  \tp_{B^+} U$ and define the right Shapovalov matrix,
which of course depends on the comultiplication on $U$. This matrix turns out to be $\tilde \Sc$ up to a twist by  the squared antipode $\tilde \gm^2$.
Recall that the latter acts by $e_\al\mapsto q^{-(\al,\al)}e_\al$ on the positive generator of root $\al$.
\begin{corollary}
  The element  $(\tilde \gm^{-2}\tp \id)(\tilde \Sc)$ is the universal right Shapovalov matrix  with respect
  to $\tilde \Delta$. In particular, it satisfies
\be
\label{Shap-intertw}
(1\tp e_\al)(\tilde \gm^{-2}\tp \tau_\al)(\tilde \Sc)= (\tilde \gm^{-2}\tp \id)(\tilde \Sc)\tilde \Delta (e_\al)
\ee
for each $\al\in \Pi$.
\end{corollary}
\begin{proof}
To verify (\ref{Shap-intertw}), apply the automorphism $\tau_\al$  to (\ref{quasi-invariance}) and get
$$
e_\al \tau_\al(\tilde S_{ij})=
\sum_{l}  \tilde S_{il}\pi^\al_{lj}q^{-(\al,\al)}q^{-h_\al}+  \tilde S_{ij}e_\al.
$$
This is a coordinate presentation of an identity in $\End(X)\tp \hat B^-$  that gives rise to a  universal form
$$
(1\tp e_\al)(\pi\tp \tau_\al)(\tilde \Sc)=(\pi\tp \id)(\tilde \Sc)(\pi\circ \tilde \gm^{2}\tp \id)\bigl(\tilde \Delta (e_\al)\bigr),
$$
whence (\ref{Shap-intertw}) follows. It  implies that $(\tilde \gm^{-2}\tp \id)(\tilde \Sc)$ is the right Shapovalov matrix.
\end{proof}
Another corollary of Proposition (\ref{Prop-quasi-inv}) is Proposition \ref{Mick_vector}, which thereby becomes a consequence  of quasi-invariance
(\ref{quasi-inv}) of the matrix $\tilde \Sc$ and invariance of the vector $\psi_X$. This explains the construction of $\hat Z(\Ac,\g)$
via Hasse diagrams presented in \cite{MS}.

Our studies demonstrate that the Shapovalov matrix has one more incarnation besides the inverse contravariant form and  intertwining operators  for Verma modules \cite{EV}.
A morphism from the category of right modules  acts as an operator in the category of left modules. The key role here belongs to
 quasi-invariance (\ref{Shap-intertw}), where zero left-hand side modulo   $\langle e_\al\rangle_{\al\in \Pi}$ is sufficient for an intertwiner.
The exact identity  (\ref{Shap-intertw}) has not been employed before, to the best of our knowlege.
We thereby conclude that  the Shapovalov matrix is a construct in its own right and should be distinguished from the inverse Shapovalov form.
\section{An alternative construction of  $\hat Z(\Ac,\g)$}
\label{secLeftShMat}
\subsection{Mickelsson algebras and left Shapovalov matrix}
\label{SecHasseDiag}
We are going to present an alternative description of the algebra $\hat Z(\Ac,\g)$ in terms of the  Shapovalov matrix $\Sc$ relative to left Verma modules.
It is a  unique element of a completed tensor product $U^+\tp \hat B^-$
 that satisfies
$$
\Delta(e_\al) \Sc\in U^+\tp \hat B^-e_\al, \quad \forall \al \in \Pi.
$$
If $V_\la$ is an irreducible Verma module with highest vector $v_\la$ and $X$ is a finite dimensional $U$-module, then the tensor
$\Sc(v\tp v_\la)$ is $U^+$-invariant for any $v\in X$, and any $U^+$-invariant tensor in $X\tp V_\la$ is obtained this way.

As well as the matrix $\tilde \Sc$, the left Shapovalov matrix  can be explicitly  expressed through the universal R-matrix:
\be
\label{geom F}
 \Sc=\sum_{n=0}^{\infty } \Sc^{(n)}, \quad \mbox{where}\quad
\Sc^{(0)}=1\tp 1,\quad \Sc^{(n+1)}=\varphi(D)\bigl( \Fc\Sc^{(n)}\bigr),
\quad n\geqslant 0.
\ee
For a graded $U^+$-module $X$ with representation homomorphism $\pi\colon U^+\to \End(X)$, the entries of the matrix
$S_X=(\pi\tp \id)(\Sc)$ in a weight basis in $X$ read
\be
 S_{ij }= \sum_{i\succ \vec m\succ j }\phi_{i,\vec m, j}A^j_{i,\vec m}.
\nn
\ee
The summation is performed over all possible routes $i\dashleftarrow j$, see \cite{M1} for details.

\newcommand{\Yc}{\mathcal{Y}}

We regard $\Ac$ as an adjoint $U$-module under the action
$$
x\tr a=\ad(x)(a)=x^{(1)}a \gm(x^{(2)}), \quad \forall x\in U, \quad \forall a\in \Ac.
$$
Suppose that  $X$ is realized  as a submodule in $\Ac$ with respect to the adjoint action of $U^+$ on $\Ac$.
Let $X^*$ be the right dual to $X$ with action $\btr$.
Define a right $U^+$-action on $X^*$  by $y\tl u=\gm(u)\btr y$ for $y\in X^*$
and consider $X^*\tp \hat \Ac$ as a right $U^+\tp \hat B^-$-module with the regular action of $B^-$ on $\hat \Ac$.

Let $\psi_{X^*}\subset X^*\tp X$ denote the invariant tensor fixed up to a scalar factor. We call it left Mickelsson generator relative to $X$.
It satisfies the identities
\be
\label{psi-invar}
\bigl(\id_X \tp  \ad(u)\bigr)(\psi_{X^*})=\psi_{X^*}(u\tp \id), \quad (\id_X \tp  u)(\psi_{X^*})=\psi_{X^*}\Delta(u), \quad \forall u\in U^+,
\ee
which are equivalent to $U^+$-invariance.
In a weight basis $\{\psi_i\}_{i\in I_X}\subset X$ and   its dual basis $\{y_i\}_{i\in I_X}\subset X^*$,
we have $\psi_{X^*}=\sum_{i\in I_X}y_{i}\tp \psi_{i}$.
 \begin{propn}
  The   vector
$$
Z_{X^*}=\psi_{X^*} \Sc (1\tp 1_\Vc)=\sum_{i,k\in I_X}y_i\tp \psi_{k}S_{ki} 1_{\hat \Vc}
$$
  belongs to $X^*\tp \hat \Vc^+$.
\end{propn}
\begin{proof}
Let us check the coordinate form of $Z_{X^*}$  given by the right equality first. Using symbolic Sweedler notation $\Sc= \Sc^+\tp \Sc^-$ we write $S_{ki}= \pi(\Sc^+)_{ki} \Sc^-$ for all $k,i\in I_X$.
Then the left formula in (\ref{psi-invar}) gives
$$
\psi_{X^*} \Sc =\sum_{i\in I_X}y_i \Sc^+\tp \psi_i \Sc^-=\sum_{i\in I_X}y_i \tp \bigl(\ad(\Sc^+)(\psi_i)\bigr) \Sc^-=
\sum_{i\in I_X}y_i \tp \sum_{k\in I_X} \pi (\Sc^+)_{ki} \psi_k\Sc^-,
$$
which immediately implies the required equality. Furthermore, for all $u\in U^+$ we have
$$
(1\tp u)Z_{X^*}=\bigl(\psi_{X^*}\Delta(u)\bigr)\Sc    =\psi_{X^*}\bigl(\Delta(u)\Sc\bigr )=\eps(u)\psi_{X^*}\Sc =\eps(u)Z_{X^*},
$$
 modulo $X^*\tp \hat \Jc_+$. We conclude that the entries of $Z_{X^*}$ are in $\hat \Vc^+$, whence the assertion follows.
\end{proof}

Thus constructing  elements of $\hat Z(\Ac,\g)$ out of $\Ac$ reduces to finding $\psi_{X^*}$ for appropriate $X$.
Remark that the invariance condition can be relaxed to
quasi-invariance
$$
(\id_X \tp  u)(\psi_{X^*})=\psi'_X\Delta(u), \quad \forall u\in U^+,
$$
for some $\psi'_X$ not necessarily equal to $\psi_{X^*}$ (and possibly depending on $u$).

\begin{example}
\em
  Suppose that $\Ac$ is a universal enveloping algebra $U(\a)$ of a Lie algebra $\a\supset \g$.
  Let $\a=X\op \g$ be a decomposition to a direct sum of modules. Put $\Zc$ to be the sum $\sum_{n=1}^{\infty}S^n(X)$ of symmetrized powers $X^n$.
  It freely generates $\Ac$ over $U$ thanks to the PBW theorem.
  It is a locally finite $\g$-module, and the canonical element $\psi_{\Zc^*}\in \Zc^*\tp \Zc$ is a universal left Mickelsson generator. Its reduction to  $X^*\tp X
  $ delivers a PBW system in $\hat Z(\a,\g)$ over $\hat U(\h)$.

\end{example}

In order to relate the above  construction with  extremal projector, consider the situation when
$X$ is an $U$-module and the element $\psi_{X^*}$ is $U$-invariant.
Let $\Theta\in U\tp \hat U_0$ be the image of $\Sc$ under the map
$$
U^+\tp \hat B_-\to U^+\tp U^-\tp \hat U^0\to U\tp \hat U^0,
$$
where the left arrow is the triangular decomposition isomorphism, and the right one acts by  $x\tp y\tp h\mapsto \gm^{-1}(y)x\tp h$.
It is known that $\Theta$ is invertible. When specialized at a generic $\la\in \h^*$, it  becomes the universal extremal twist  \cite{M3};
it is
 also related with dynamical Weyl group \cite{EV}.
Now we derive the action of $\wp$  on the components of vector $\psi_{X^*}$:
\be
\label{Mick-el}
Z_{X^*}=(\id \tp\wp) \psi_{X^*}  \Theta = \psi_{X^*} S_X
\ee
modulo $X^*\tp \Jc_+$. Thus the left Shapovalov matrix does not provide immediate evaluation of the $\wp$-action, but through
the extremal twist $\Theta$. This is a distinction from the approach based on the right Shapovalov matrix.
\subsection{Quantum Lax operators}
In this section we construct left Mickelsson generators for the case of quantum reductive pairs.
Suppose that $\a$ is a simple Lie algebra and $\g\subset \a$ is the commutant of a Levi subalgebra.
Let $\k\supset \h$ denote the Cartan subalgebra of $\a$,  $\a_\pm \supset \g_\pm$the maximal nilpotent subalgebras,
and $\c$ the orthogonal complement of $\h$ in $\k$.
The quantum group $U_q(\g)$ is a natural Hopf subalgebra in $U_q(\a)$, for which   the  conventions of Section \ref{SecQuantGroups}
are in effect.

It is known that $U_q(\a)$ contains a locally finite $U_q(\g)$-submodule $\Zc$ such  that
$$
U_q(\a)=U_q(\g_-)\Zc U_q(\c) U_q(\b_+).
$$
It can be constructed from quantized nilradicals of the parabolic subalgebras $\g+\c+\a_\pm$. They are locally finite as $U_q(\g)$-modules and their existence is proved in \cite{Ke}.
The corresponding Mickelsson generator $\psi_{\Zc^*}$ gives rise to   the entire $\hat Z(\a,\g)$, as a module over $\hat U_q(\k)$.
Next we address the question   of PBW basis in $\hat Z(\a,\g)$.

By $\Ru$ we now understand the universal R-matrix of $U_q(\a)$, and $\check{\Ru}$ is obtained from it similarly as in the case of $U_q(\g)$ considered before.
The intertwining identities for $\Ru$ translate to the following relations for the matrix $\check{\Ru}$.
\be
 \check{R}(e_{\al}\tp q^{h_{\al}} + 1\tp e_{\al})&=& (e_{\al}\tp q^{-h_{\al}} + 1\tp e_{\al})\check{R},
\label{IRe}
\\
\check{R}(f_\al\tp 1 + q^{-h_\al}\tp f_\al)&=&(q^{h_\al}\tp f_\al + f_\al \tp 1)\check{R},
\label{IRf}
\\
(e_{\al}\tp q^{h_{\al}} + 1\tp e_{\al}) \check{R}^{-1}&=& \check{R}^{-1}(e_{\al}\tp q^{-h_{\al}} + 1\tp e_{\al}),
\label{IReInv}
\\
(f_\al\tp 1 + q^{-h_\al}\tp f_\al) \check{R}^{-1}&=& \check{R}^{-1} (q^{h_\al}\tp f_\al + f_\al \tp 1),
\label{IRfInv}
\ee
Here $\al$ is a simple root from $\Pi_\a$, although we need these identities only for $\al\in \Pi_\g$.
Note that $\check{\Ru}^{-1}\in U_q(\a_+)\tp U_q(\a_-)$.

Recall that, for the given comultiplication $\Delta$, the adjoint action on the generators of the quantum group
explicitly reads
$$
\ad(e_\al) u= e_\al u q^{-h_\al } - u  e_\al q^{-h_\al},
\quad
\ad(f_\al) u=f_\al u - q^{-\left(\al,\wt(u)\right)} u f_\al,
$$
for all $\al\in \Pi_\g$ and $u\in U_q(\a)$.

Fix a $U_q(\a)$-module $V$ with representation homomorphism $\varrho\colon U_q(\a)\to \End(V)$ and consider the quantum
Lax operators
$$
L^-_X=(\varrho\tp \id)(\check{\Ru}^{-1}) \in \End(V)\tp U_q(\a_-),
\quad
L^+_X=(\id \tp \varrho)(\check{\Ru})\in U_q(\a_+)\tp \End(V).
$$
Suppose that  $V^\g\not \not =\{0\}$ and pick  a   $U_q(\g)$-invariant non-zero weight vector $v_0\in V^\g$.
Define
$$
\psi^-_{i}=L^-_{i0} ,\quad
\psi^+_{i}=L^+_{i0}q^{h_{\nu_i}}.
$$
These elements of $U_q(\a)$ carry weights  $\wt(\psi^-_{i})=\nu_0-\nu_i< 0$ and  $\wt(\psi^+_{i})=\nu_0-\nu_i> 0$.
The weight $\nu_0$ is orthogonal to $\Pi_\g$.
\begin{propn}
\label{psi-adjoint}
For each $v_0\in V^\g$ the vector spaces $\Span\{\psi^-_{i}\}_{0\prec i}$ and $\Span\{\psi^+_{i}\}_{i\prec 0}$ are invariant under the adjoint action of $U_q(\g)$
on  $U_q(\a)$. Specifically,
\be
\ad(e_\al)(\psi^+_{i})&=& - \sum_{k\prec 0}q^{(\al,\nu_i)}\varrho(e_{\al})_{ik}   \psi^+_{k}=\sum_{k\prec 0}\varrho\bigl(\tilde \gm^{-1}(e_{\al})\bigr)_{ik}\psi^+_{k},
\nn\\
 \ad(f_\al)(\psi^+_{i})&=&  -\sum_{k\prec 0} \varrho(f_\al)_{ik} q^{-(\al,\nu_k)} \psi^+_{k}=\sum_{k\prec 0} \varrho\bigl(\tilde \gm^{-1}(f_\al)\bigr)_{ik} \psi^+_{k},
 \nn\\
\ad(e_\al)(\psi^-_{i})&=&  -\sum_{0\prec k}\varrho(e_{\al})_{ik}q^{-(\al,\nu_k)}\psi^-_{k}=  \sum_{0\prec k}\varrho\bigl( \gm(e_{\al})\bigr)_{ik}\psi^-_{k},
\nn\\
\ad(f_\al)(\psi^-_{i})&=&   -\sum_{0\prec k} q^{(\al,\nu_i)}\varrho(f_\al)_{ik} \psi^-_{k}= \sum_{0\prec k} \varrho\bigl( \gm(f_\al)\bigr)_{ik} \psi^-_{k}.
\nn
\ee
\end{propn}
 \begin{proof}
 This is obvious with regard to the subalgebra $U^0$ as the elements $\psi^\pm_i$ carry definite weights. To complete the proof, send to the representation the appropriate tensor factor in
(\ref{IRe}-\ref{IRfInv}) and take into account that $\varrho(f_\al)_{k0}=\varrho(e_\al)_{k0}=(\al,\nu_0)=0$ for all $\al\in \Pi_\g$ and all $k\in I_V$.
\end{proof}
Under the assumptions of Proposition \ref{psi-adjoint}, one can construct Mickelsson generators by restricting  the representation $\varrho$ to  $U_q(\g)$ and take $X=\Span\{\psi^-_{i}\}_{0\prec i}$  with the representation homomorphism $\pi=\varrho^t\circ \gm$
or $X=\Span\{\psi^+_{i}\}_{i\prec 0}$ with $\pi=\varrho^t\circ \tilde\gm^{-1}$. Here  $t$ stands for the matrix transposition.

In order to construct a PBW basis in $\hat Z(\a,\g)$, one can take for  $V$  the irreducible finite dimensional $U_q(\a)$-module $\a_q$ whose highest weight is the maximal root.
Let $\g_q\subset \a_q$ be the analogous $U_q(\g)$-module.
  The sum $\sum_{\al\in \Rm_\a\backslash \Rm_\g} \a_q[\al]$ forms a self-dual $U_q(\g)$-submodule,
$\a_q/(\a_q[0]+\g_q)$. In the classical limit, it goes over to $\a/(\c\op \g)$ whose elements generate a PBW basis in $U(\a)/\bigl(\g_-U(\g)+U(\g)\g_+\bigr)$
as a $U(\k)$-module. Each irreducible $U_q(\g)$-submodule in $\a_q/(\a_q[0]+\g_q)$ can be constructed along the lines of Proposition \ref{psi-adjoint}
by picking an appropriate  vector $v_0\in \a_q^\g[0]$. Taking the direct sum of such modules for $X$ one arrives at a Mickelsson
generator $\psi_{X^*}$. Multiplying it by left Shapovalov matrix on the right yields a vector whose components generates a PBW basis
in each weight of $\hat Z(\a,\g)$, for generic $q$, or equivalently, over $\C[[\hbar]]$.
With regard to a particular $q$, the problem  reduces to the question if the entries of L-operators generate a PBW basis in $U_q(\a)$.

 \vspace{20pt}
 \vspace{20pt}
\noindent

\vspace{20pt}
\noindent

\underline{\large \bf Acknowledgement}
\vspace{10pt} \\
This work was done at the Center of Pure Mathematics MIPT.
It is financially supported  by Russian Science Foundation  grant 23-21-00282.

\subsection*{Declarations}
\subsubsection*{Data Availability}
 Data sharing not applicable to this article as no datasets were generated or analysed during the current study.
\subsubsection*{Funding}
This research is  supported  by Russian Science Foundation  grant 23-21-00282.
\subsubsection*{Competing interests}
The authors have no competing interests to declare that are relevant to the content of this article.

\end{document}